\documentclass[a4paper,11pt]{article}
\usepackage[top=2.5cm,bottom=2.5cm,left=2.2cm,right=2.2cm]{geometry}
\usepackage{amsfonts}
\usepackage{mathrsfs,amscd,amssymb,amsthm,amsmath,bm,graphicx,psfrag,subfigure,url,mathtools}
\usepackage{pict2e}
\usepackage{stfloats}
\usepackage{psfrag,amsmath}
\usepackage{tikz}
\usepackage{indentfirst}
\usepackage{hyperref}
\usepackage{bookmark}
\usepackage{booktabs}
\usepackage{enumerate}
\usepackage{latexsym,euscript,epic,eepic,color}
\usepackage{multirow}
\usepackage{multicol}
\usepackage{longtable}
\usepackage{adjustbox}
\usepackage[all]{xy}
\usepackage{setspace}
\usepackage{epstopdf}
\allowdisplaybreaks
\usepackage{authblk}

\makeatletter

\renewcommand{\@seccntformat}[1]{{\csname the#1\endcsname}{\normalsize .}\hspace{.5em}}
\makeatother

\usepackage{ifpdf}
\usepackage{indentfirst}

\def \[{\begin{equation}}
\def \]{\end{equation}}

\newtheorem{thm}{Theorem}[section]

\newtheorem{claim}{Claim}
\newtheorem{case}{Case} 

\newtheorem{problem}{Problem} 

\newtheorem{remark}{Remark}
\newtheorem{fact}{Fact}

\newenvironment{wst}
{\setlength{\leftmargini}{1.5\parindent}
 \begin{itemize}
 \setlength{\itemsep}{-1.1mm}}
{\end{itemize}}
\begin{document}
\baselineskip=0.23in
\title{\bf 	
	A generalization of the Chv\'{a}tal-Erd\H{o}s theorem\thanks{
{\it Email address}: chengkunmath@163.com (K. Cheng)}}
\author{Kun Cheng}
\affil{Department of Mathematics, East China Normal University, Shanghai, 200241, China}
\date{\today}
\maketitle
\begin{abstract}
	A well-known result of Chv\'{a}tal and Erd\H{o}s from 1972 states that a graph  with connectivity  not less than its independence number plus one  is hamiltonian-connected.
A graph $G$ is called an $[s,t]$-graph if any induced subgraph of $G$ of order $s$ has size at least $t.$ We prove that every $k$-connected $[k+1,2]$-graph is hamiltonian-connected except $kK_1\vee G_{k},$ 
where $k\ge 2$ and $G_{k}$ is an arbitrary graph of order $k$.  
This generalizes the Chv\'{a}tal-Erd\H{o}s theorem. 
 \vspace{0.2cm}

\end{abstract} 

\section{\normalsize Introduction}
We consider finite simple graphs and use standard terminology and notation from \cite{bondy1} and \cite{West}.    
Let $G=(V(G), E(G))$ be a graph with vertex set $V(G)$ and edge set $E(G)$. Then $n(G):=|V(G)|$ and $e(G):=|E(G)|$ are called the {\it order} and {\it size} of $G$, respectively.  
For $u, v\in V(G)$, we denote by $u\sim v$ if $u$ and $v$ are adjacent, and $u\not\sim v$ otherwise. For $v\in V(G)$, let $N_G(v)$ and $d_G(v)$ be the neighborhood and the degree of $v$ in $G$, respectively.
For $S\subseteq V(G)$, we denote by $G[S]$ the subgraph of $G$ induced by $S$, 
and let $e(S):=e(G[S])$ for short.  
We write $N_S(v)$ for $N_G(v)\cap S$ and let $d_S(v):=|N_S(v)|$. 
For a subgraph $H$ of $G$, let $N_H(v)=N_G(v)\cap V(H)$. 
Let $X,Y\subseteq V(G)$ be two disjoint vertex sets, we use $e(X,Y)$ to denote the number of edges with one end in $X$ and the other in $Y$. 
In particular, if $X=\{x\}$ is singleton, we usually write $e(x,Y)$  instead of $e(\{x\},Y)$.
We write $P_n$, $C_n$ and $K_n$ for the path, the cycle and the complete graph of order $n$, respectively.  
For two graphs $G$ and $H,$ $G\vee H$ denotes the {\it join} of $G$ and $H,$ which is obtained from the disjoint union $G+H$ by adding edges joining every vertex of $G$ to every vertex of $H.$  
For graphs we will use equality up to isomorphism, so $G = H$ means that $G$ and $H$ are isomorphic.

Denote by $\alpha (G)$ the independence number of a graph $G$.  
For two distinct vertices $x$ and $y$ in $G$, an {\em $(x, y)$-path}  is a path whose endpoints are $x$ and $y$.  
A Hamilton path (resp., cycle) in $G$ is a path (resp., cycle) containing every vertex of $G$.  
$G$ is {\em hamiltonian} if it contains a Hamilton cycle.  
A graph is called  {\em  hamiltonian-connected} if between any 
two distinct vertices there is a Hamilton path.
In 1972, Chv\'{a}tal and Erd\H{o}s~\cite{CE1} proved  the following result.
\begin{thm}[Chv\'{a}tal-Erd\H{o}s~\cite{CE1}]\label{thmCE}
	Let $k\ge 2$ be an integer.  If $G$ is a $k$-connected graph with $\alpha(G)\le k-1,$ then $G$ is hamiltonian-connected.
\end{thm}
Theorem~\ref{thmCE} is an important result in hamiltonian graph theory, 
which opened an entirely new avenue for investigation. We refer the reader to \cite{ChenGT, Chen2,Faud,Jack, Saito}.

Let $s$ and $t$ be given integers. A graph $G$ is called an {\it $[s,t]$-graph} if any induced subgraph of $G$ of order $s$ has size at least $t.$  
Clearly, $\alpha(G)\le k$ if and only if
$G$ is a $[k+1,1]$-graph. Thus the concept of an $[s,t]$-graph is an extension of the independence number.  Wang and Mou~\cite{Wang1} showed that every $[s,t]$-graph is a $[s+1,t+1]$-graph. 
For more advances in this topic, we refer the reader to \cite{Wang2, LiuWang, Zhan2} and  the references therein.


In 2015, Wang and Mou~\cite{Wang1} proved that every $k$-connected $[k+2,2]$-graph is hamiltonian, except Petersen graph and $(k+1)K_1\vee G_k,$ where $k\ge 1$ and  $G_k$ is an arbitrary graph of order $k.$
Motivated by their result, we prove the following.
\begin{thm}\label{thm1}
Let $k\ge 2$ be an integer. If $G$ is a $k$-connected $[k+1,2]$-graph, 
then $G$ is hamiltonian-connected unless $G= kK_1\vee G_{k},$ where $G_k$ is an arbitrary graph of order $k$.  
\end{thm}
\begin{remark} 
If a graph $G$ satisfies the conditions of Theorem~\ref{thmCE}, then $G$ is a $k$-connected $[k,1]$-graph, and hence a $[k+1,2]$-graph. Note that  $kK_1\vee G_{k}$ has  independence number $k$.  
Thus Theorem~\ref{thm1} generalizes Theorem~\ref{thmCE}.
\end{remark}
The remaining sections are organized as follows.
We prove Theorem~\ref{thm1} in Section~\ref{s2},  
and give some concluding remarks in Section~\ref{s3}. 

\section{\normalsize Proof of Theorem~\ref{thm1}}\label{s2}
Assume that the graph $G$ satisfies the conditions of Theorem~\ref{thm1} and $G$ is not hamiltonian-connected.  We are to show that  $G= kK_1\vee G_{k}.$

When $k=2,$ $G$ is a $2$-connected $[3,2]$-graph. One sees $\alpha(G)\le 2.$
If $G$ is $3$-connected, by Theorem~\ref{thmCE}, $G$ is hamiltonian-connected, a contradiction. 
Thus, $G$ contains a vertex cut, say $S$, with cardinality 2.
If $n\ge 5$, taking $v_1,v_2,v_3$ not all from the same components of $G-S$, one has $e(G[\{v_1,v_2,v_3\}])\le 1$, which contradicts the fact that $G$ is a $[3,2]$-graph. 
Therefore, $n\le 4$. 
Since $G$ is 2-connected, 
$G\in\{ K_3,K_4, C_4,  K_4-e\}$. 
Since the complete graph is hamiltonian-connected, $G\in\{2K_1\vee 2K_1, 2K_1\vee K_2\},$  as desired. 

We now consider the case $k\ge 3.$ 
Note that there exist two vertices $u,v\in V(G)$ such that there exists no Hamilton $(u,v)$-path in $G.$ 
Let $P=v_0v_1\ldots v_l$ 
be a longest $(u,v)$-path, where $u=v_0$ and $v=v_l.$ 
Let 
\begin{align*}
	v_i^{+p}=v_{i+p} \text{~~and~~} 
	v_i^{-p}=v_{i-p}.
\end{align*}
Set $v_i^+=v_{i+1}$ and $v_i^-=v_{i-1}$ for short. 
If $a<b,$
we set 
\begin{align*}
	v_a \overrightarrow{P} v_b:=v_a v_{a+1}\ldots  v_b \text{~~and~~} 
	v_b \overleftarrow{P} v_a:=v_b v_{b-1}\ldots  v_a.
\end{align*}

Note that $G-V(P)\neq \emptyset$. For a component $H$ of $G-V(P)$, 
let 
$$N_P(H)=\bigcup_{v\in V(H)} N_P(v).$$
One has $|N_P(H)|\ge k$  since $G$ is $k$-connected. 
We write 
\begin{align*}
N_P^+(H):=\{w^+ : w\in N_P(H) \setminus \{v_l\}\},~~
N_P^-(H):=\{w^- : w\in N_P(H) \setminus \{v_0\}\}.
\end{align*}

We proceed with a sequence of claims, the first one is a direct consequence of the fact that $P$ is a longest $(v_0,v_l)$-path.
\begin{claim}\label{c1}
\begin{wst}
\item[{\rm (i)}]
If $w\in N_P(H)$, then $\{w^-,w^+\} \cap N_P(H)=\emptyset;$ 
\item[{\rm (ii)}]
$N_P^+(H)$ (resp., $N_P^-(H)$) is an independent set of $G$.
\end{wst}
\end{claim}

\begin{claim}\label{c2}
$|G-V(P)|=1$.
\end{claim}
\begin{proof}[Proof of Claim~\ref{c2}]
Choose $y\in V(H)$. Suppose to the contrary that $|G-V(P)|\ge 2$. 
Then there exists another vertex $y'\in V(G)\setminus V(P).$ 
Setting $U:=\{y,y'\}\cup N_P^+(H)$, one has $|U|\ge k+1.$ 

If $y'\in V(H)$, by Claim~\ref{c1}(i), $y$ (resp., $y'$) is nonadjacent to any vertex in $N_P^+(H)$. Combining with Claim~\ref{c1}(ii) gives $e(U)\le 1$, which contradicts the fact that $G$ is a $[k+1,2]$-graph.

Thus, $y'\not\in V(H)$ and $|X|=1$ for any component $X$ of $G-V(P).$ We assert that $y'$ is adjacent to at most one vertex in $N_P^+(y)$. 
If not, suppose that $y'\sim v_{a+1}$ and $y'\sim v_{b+1}$ where $v_a,v_b\in N_P(y)$ and $a<b.$ 
Then $v_0\overrightarrow{P}v_a y v_b \overleftarrow{P} v_{a+1} y' v_{b+1}\overrightarrow{P} v_l$ is a $(v_0,v_l)$-path longer than $P,$ which is impossible. 
Now $U=\{y,y'\}\cup N_P^+(y)$ and 
$e(U)$ is again at most one, a contradiction. This proves Claim~\ref{c2}.
\end{proof}

From Claim~\ref{c2}, by $y$ we denote the only vertex of $G-V(P).$

\begin{claim}\label{c3}
$d_P(y)=k$.
\end{claim}
\begin{proof}[Proof of Claim~\ref{c3}]
Note that $d_P(y)\ge k$. Suppose $d_P(y)\ge k+1.$ 
Setting $A:=\{y\}\cup N_P^+(y)$, we get $|A|\ge k+1$ and $e(A)=0$, a contradiction.
\end{proof}
By Claim~\ref{c3}, we may write $N_P(y)=\{v_{i_1},\ldots,v_{i_k}\},$ with $0\le i_1<\cdots<i_k\le l$. 
\begin{claim}\label{c4}
\begin{wst}
\item[{\rm (i)}]
$i_1=0$ and $i_k=l$;
\item[{\rm (ii)}] 
$N_{v_{a+2} \overrightarrow{P} v_b}^-(v_{a+1})\cap N(v_{b+1})=\emptyset$ and $N_{v_{a+1} \overrightarrow{P} v_{b-1}}^-(v_{a-1})\cap N(v_{b-1})=\emptyset$,   for $a,b\in \{i_1,\ldots, i_k\}$ with $a<b$; 
\item[{\rm (iii)}] 
$N_{v_a \overrightarrow{P} v_b}^+(v_{a-1})\cap N(v_{b+1})=\emptyset$, for $a,b\in \{i_2,\ldots, i_k\}$ with $a<b$;
\item[{\rm (iv)}] 
$N_{v_b \overrightarrow{P} v_l}^+(v_{a-1})\cap N(v_{b-1})=\emptyset$, for $a,b\in \{i_2,\ldots, i_{k-1}\}$ with $a<b$;
\item[{\rm (v)}] 
If $w\in N_P(v_{a-1})$, then $w^+\not\sim v_{a-1},$ where $a\in \{i_2,\ldots, i_k\}.$
\end{wst}
\end{claim}
\begin{proof}[Proof of Claim~\ref{c4}]
(i) 
Suppose to the contrary that $i_1\ge 1$. Then $v_{i_1}^-$ exists, and hence $N_P^-(y)\cup \{y\}$ is an independent set of cardinality $k+1$ by Claim~\ref{c1}, a contradiction. This implies $i_1=0$, 
  and similarly, $i_k=l$.

(ii)  
To the contrary, suppose that there exists a vertex, say $v_t$, in $N_{v_{a+2} \overrightarrow{P} v_b}^-(v_{a+1})\cap N(v_{b+1}).$ 
Then $v_0 \overrightarrow{P} v_a y v_b \overleftarrow{P} v_{t+1} v_{a+1} \overrightarrow{P} v_t v_{b+1} \overrightarrow{P} v_l$ is a Hamilton $(v_0,v_l)$-path in $G$, a contradiction.  
By the symmetry of $\overrightarrow{P}$ and $\overleftarrow{P}$,  
we get
$N_{v_{a+1} \overrightarrow{P} v_{b-1}}^-(v_{a-1})\cap N(v_{b-1})=\emptyset$.

(iii)  
Suppose that there exists a vertex, say $v_s$, in $N_{v_a \overrightarrow{P} v_b}^+(v_{a-1})\cap N(v_{b+1}).$ 
Then $$v_0 \overrightarrow{P} v_{a-1} v_{s-1} \overleftarrow{P} v_a y v_b \overleftarrow{P} v_s v_{b+1} \overrightarrow{P} v_l$$ is a Hamilton $(v_0,v_l)$-path in $G$, a contradiction.

(iv)  
Suppose to the contrary that  there exists a vertex $w\in N_{v_b \overrightarrow{P} v_l}^+(v_{a-1})\cap N(v_{b-1}).$ 
Then $$v_0 \overrightarrow{P} v_{a-1} w^- \overleftarrow{P} v_b y v_a \overrightarrow{P} v_{b-1} w \overrightarrow{P} v_l$$ is a Hamilton $(v_0,v_l)$-path in $G$, a contradiction.

(v)  
Suppose to the contrary that there exists a vertex $w\in N_P(v_{a-1})$ 
such that $w^+\sim v_{a-1}.$ 
Then $w\in V(v_0 \overrightarrow{P} v_{a-3}) \cup 
V(v_a  \overrightarrow{P}  v_{l-1}).$
Note that $a\in \{i_2,\ldots, i_k\}.$ 
Then $v_{a-2}$ exists. 
If $v_{a-2}\sim y$, 
then one of the following $(v_0,v_l)$-path is longer than $P$:
\begin{align*}
& v_0 \overrightarrow{P} v_{a-2} y v_a \overrightarrow{P} w v_{a-1} w^+ \overrightarrow{P} v_l \text{~(if $w\in V(v_a \overrightarrow{P} v_{l-1})$)},\\
& v_0 \overrightarrow{P} w v_{a-1} w^+ \overrightarrow{P} v_{a-2} y v_a \overrightarrow{P} v_l \text{~(if $w\in V(v_0 \overrightarrow{P} v_{a-3})$)},
\end{align*}
a contradiction. Therefore $v_{a-2}\not\sim y.$ 

Let $U:=\{y,v_{a-2}\} \cup N_P^-(y).$ Clearly, $|U|=k+1$, then $e(U)\ge 2.$ 
By Claim~\ref{c1}, $\{y\} \cup N_P^-(y)$ is an independent set of $G,$ 
it follows that $e(v_{a-2}, N_P^-(y))\ge 2.$ 
Hence there exists a subscript $b\in\{i_2,\ldots, i_k\} \setminus \{ a\}$ such that $v_{a-2}\sim v_{b-1}.$

If $w\in V(v_a \overrightarrow{P} v_{l-1}),$ 
Since $v_{a-2}\not\sim y$, by Claim~\ref{c1},  $v_b\not\in\{v_{a-2},   v_{a-1}, v_{a+1}, w^+,w^{+2}\}.$   
Now we get one of the following $(v_0,v_l)$-path longer than $P$: 
\begin{align*}
& v_0 \overrightarrow{P} v_{b-1} v_{a-2} \overleftarrow{P} v_b y v_a \overrightarrow{P} w v_{a-1} w^+ \overrightarrow{P} v_l 
\text{~(if  $v_b\in V(v_{i_2} \overrightarrow{P} v_{a-3})$)},\\
& v_0 \overrightarrow{P} v_{a-2} v_{b-1} \overleftarrow{P} v_a y v_b \overrightarrow{P} w v_{a-1} w^+ \overrightarrow{P} v_l \text{~(if $v_b\in V(v_{a+2} \overrightarrow{P} w)$)},\\
& v_0 \overrightarrow{P} v_{a-2} v_{b-1} \overleftarrow{P} w^+ v_{a-1} w \overleftarrow{P} v_a y v_b \overrightarrow{P} v_l \text{~(if $v_b\in V(w^{+3} \overrightarrow{P} v_l)$)}.
\end{align*}
This contradicts the choice of $P.$

Then we assume that $w \in V(v_0 \overrightarrow{P} v_{a-3}).$ 
Recall that $v_{a-2}\not\sim y$. By Claim~\ref{c1}, $v_b\not\in\{w^+, w^{+2},v_{a-2}, \allowbreak  v_{a-1},\allowbreak  v_{a+1} \}$. 
Now we get one of the following $(v_0,v_l)$-path longer than $P$: 
\begin{align*}
	& v_0 \overrightarrow{P} v_{b-1} v_{a-2} \overleftarrow{P} w^+  v_{a-1}  w \overleftarrow{P} v_b y v_a \overrightarrow{P} v_l \text{~(if $v_b\in V(v_{i_2} \overrightarrow{P} w)$)},\\
	& v_0  \overrightarrow{P} w v_{a-1} w^+ \overrightarrow{P} v_{b-1} v_{a-2} 
	\overleftarrow{P} v_b y v_a   \overrightarrow{P} v_l \text{~(if $v_b\in V(w^{+3} \overrightarrow{P} v_{a-3})$)},\\
	& v_0 \overrightarrow{P} w v_{a-1} w^+ \overrightarrow{P} v_{a-2} v_{b-1}  \overleftarrow{P} v_a y v_b \overrightarrow{P} v_l \text{~(if $v_b\in V(v_{a+2} \overrightarrow{P} v_l)$)}.
\end{align*}
This also contradicts the choice of $P,$ which proves (iv). 
The proof of Claim~\ref{c4} is complete.
\end{proof}

Now, we are in a position to present the proof of Theorem~\ref{thm1} for $k\ge 3$.
	Recall that $y$ is the only vertex of $G-V(P)$ and  $N_P(y)=\{v_{i_1},\ldots,v_{i_k}\}.$ 
	We derive the proof into the following two cases.

\begin{case}\label{cs1}
 There exists an integer $j\in \{1,\ldots,k-1\},$ such that $|v_{i_j}^+\overrightarrow{P} v_{i_{j+1}}^-|=1.$
\end{case}

By the symmetry of $\overrightarrow{P}$ and $\overleftarrow{P}$,  
we assume that $1\le j\le k-2,$ and so 
$v_{i_{j+2}}$ exists. 
Now we proceed with the following two facts in this case, which give the structure of the desired graph.
\begin{fact}\label{f1}
  $N(v_{i_j}^+)= N_P(y).$   
\end{fact}
\begin{proof}[Proof of Fact~\ref{f1}]
By Claim~\ref{c1}, $N(v_{i_j}^+)\subseteq V(P).$ 
We assert that $N(v_{i_j}^+)\subseteq N_P(y).$ 
Suppose to the contrary that $N(v_{i_j}^+)\not\subseteq N_P(y).$ 
There exists a vertex $x\in N(v_{i_j}^+) \setminus N_P(y).$ 
By the symmetry of $\overrightarrow{P}$ and $\overleftarrow{P}$,   
we may assume that $x\in V(v_{i_{j+1}}^+\overrightarrow{P} v_l).$ 
Since $v_l \in N_P(y),$ we have $x\neq v_l$ and so $x^+$ exists. 
Note that $|v_{i_j}^+\overrightarrow{P} v_{i_{j+1}}^-|=1.$ 
By Claim~\ref{c1}, $x^-\not\sim y$ and $x^+\not\sim y.$ 
Let $U:=\{y,x^+\}\cup N_P^+(y).$ 
One sees $|U|=k+1.$ 
Hence, $e(U)\ge 2.$ 
This implies that $e(x^+, N_P^+(y))\ge 2.$  
Hence, there exists a vertex $v_{i_t}\in N_P (y)$ such that $x^+\sim v_{i_t}^+$. 
Since $v_{i_j}^+ = v_{i_{j+1}}^-,$
from Claim~\ref{c4}(v), 
$v_{i_j}^+\not\sim x^+.$ 
By Claim~\ref{c4}(iii), 
$v_{i_t}\not\in V(x^+ \overrightarrow{P} v_l).$ 
If 
$v_{i_t} \in V(v_{i_{j+1}} \overrightarrow{P} x^{-2}),$ then 
$$v_0 \overrightarrow{P} v_{i_j} y v_{i_t} \overleftarrow{P} v_{i_j}^+ x \overleftarrow{P} v_{i_t}^+ x^+ \overrightarrow{P} v_l$$ 
is a Hamilton $(v_0,v_l)$-path in $G$, a contradiction. 
Thus  
$v_{i_t} \in V(v_0 \overrightarrow{P} v_{i_j}),$ 
which gives a Hamilton $(v_0,v_l)$-path in $G$, that is, 
$$v_0 \overrightarrow{P} v_{i_t} y v_{i_{j+1}} \overrightarrow{P} x v_{i_j}^+ \overleftarrow{P} v_{i_t}^+ x^+ \overrightarrow{P} v_l,$$ a contradiction. 
So we do indeed have $N(v_{i_j}^+)\subseteq N_P(y).$ 
Since $G$ is $k$-connected and $|N_P(y)|=k$, we have $N(v_{i_j}^+)=N_P(y).$
This proves Fact~\ref{f1}.
\end{proof}

\begin{fact}\label{f2}
$|v_{i_{j+1}}^+\overrightarrow{P} v_{i_{j+2}}^-|=1.$  
\end{fact}

\begin{proof}[Proof of Fact~\ref{f2}]
Suppose to the contrary that $|v_{i_{j+1}}^+\overrightarrow{P} v_{i_{j+2}}^-|\ge 2.$ 
Let $X:=\{y, v_{i_{j+2}}^-\}\cup N_P^+(y).$ 
Since $|X|=k+1$, $e(X)\ge 2.$ 
By Claim~\ref{c1}, $e(v_{i_{j+2}}^-, N_P^+(y))\ge 2.$
Then there exist two integer $s$ and $t$, such that $v_{i_{j+2}}^- \sim v_{i_s}^+$ and $v_{i_{j+2}}^- \sim v_{i_t}^+.$ 
Clearly, one of $\{s,t\}$ is not $j+1.$ 
Without loss of generality, we may assume that $t\neq j+1.$ 
If $t<j$, by Fact~\ref{f1}, 
$v_{i_j}^+\sim v_{i_t}$, and hence 
$$v_0 \overrightarrow{P} v_{i_t} v_{i_j}^+ \overleftarrow{P} v_{i_t}^+ v_{i_{j+2}}^- \overleftarrow{P} v_{i_{j+1}} y v_{i_{j+2}} \overrightarrow{P} v_l$$ 
is a Hamilton $(v_0,v_l)$-path in $G$, a contradiction. 
If $t=j,$ by Fact~\ref{f1}, $v_{i_j}^+\sim v_{i_{j+2}},$ 
which contradicts Claim~\ref{c4}(v). 
Thus, we assume that $t\ge j+2.$ 
From Fact~\ref{f1}, 
$v_{i_j}^+\sim v_{i_t}$,  
and so 
$$v_0 \overrightarrow{P} v_{i_j} y v_{i_{j+2}} \overrightarrow{P} v_{i_t} v_{i_j}^+ \overrightarrow{P} v_{i_{j+2}}^- v_{i_t}^+ \overrightarrow{P} v_l$$  
is a Hamilton $(v_0,v_l)$-path in $G,$ a contradiction. 
This proves Fact~\ref{f2}.
\end{proof}

By the symmetry of $\overrightarrow{P}$ and $\overleftarrow{P}$, combine with Facts~\ref{f1} and \ref{f2}, 
one has $|v_{i_p}^+\overrightarrow{P} v_{i_{p+1}}^-|=1$ 
and $N(v_{i_p}^+)= N_P(y)$ 
for each $p\in \{1,\ldots,k-1\}.$ 
It means that $G= kK_1\vee G_{k},$ where $G_{k}$ is an arbitrary graph of order $k$, as desired.

\begin{case}
$\min\{|v_{i_j}^+\overrightarrow{P} v_{i_{j+1}}^-| : 1\le j\le k -1 \}  \ge 2.$
\end{case}

For the longest $(u,v)$-path $P$ in $G$, 
we denote 
$$\rho(P):=  \max\{|v_{i_j}^+\overrightarrow{P} v_{i_{j+1}}^-| : 1\le j\le k-1  \}.$$
Choose a longest $(u,v)$-path $P'$  such that 
$\rho(P')$ is as large as possible.  
Then Claims~\ref{c1}-\ref{c4} also hold for $P'.$  
In what follows, we denote $P'=u_0 u_1\ldots u_l$, where $u=u_0$ and $v=u_l.$ 
Let $z$ be the only vertex of $G-V(P')$  and write $N_{P'}(z)=\{u_{i_1},\ldots,u_{i_k}\},$ with $0\le i_1<\cdots<i_k\le l$.  
We assume that  $\rho(P')= |u_{i_q}^+\overrightarrow{P'} u_{i_{q+1}}^-|.$ 
Since  $\min\{|u_{i_j}^+\overrightarrow{P'} u_{i_{j+1}}^-| : 1\le j\le k-1  \} =1$  has been treated in Case~\ref{cs1}, 
we assume that 
$\min\{|u_{i_j}^+\overrightarrow{P'} u_{i_{j+1}}^-| : 1 \le j\le k-1  \} \ge 2.$ 

By the symmetry of $\overrightarrow{P'}$ and $\overleftarrow{P'}$,   we may assume that $2\le q\le k-1$,  
so both $u_{i_{q}}^-$ and $u_{i_{q}}^{-2}$ exist.
\begin{fact}\label{f3}
		$N_{ u_{i_q}^+\overrightarrow{P'} u_{i_{q+1}}  }  (u_{i_{q}}^- ) =\emptyset$.
\end{fact}
\begin{proof}[Proof of Fact~\ref{f3}]
We first show that  $N_{ u_{i_q}^+\overrightarrow{P'} u_{i_{q+1}}^-  }  (u_{i_{q}}^- ) =\emptyset$.	 
Suppose to the contrary that there exists a vertex $x\in N_{ u_{i_q}^+\overrightarrow{P'} u_{i_{q+1}}^-  }  (u_{i_{q}}^- )$. 
By Claim~\ref{c1} and Claim~\ref{c4}(v), $x\not\in\{  u_{i_q}^+,  u_{i_{q+1}}^-\}$. 
Setting $U=\{ z, x^- \}\cup N_{P'}^-(z).$ By Claim~\ref{c4}(v), 
$x^-\not\sim u_{i_q}^-.$ 
Combining with Claim~\ref{c4}(ii) and (iv) gives that $x^-$ is nonadjacent to any vertex of $ N_{P'}^-(z).$ 
Therefore, $e(U)=0$, which contradicts the fact that $G$ is a $[k+1,2]$-graph. 

We then prove that $u_{i_{q}}^-\not\sim  u_{i_{q+1}}$. 
Otherwise, $u_{i_{q}}^-\sim  u_{i_{q+1}},$ then we write $W:= \{z, u_{i_{q+1}}^- \}\cup  N_{P'}^+(z)$. 
Suppose that there exists a vertex $w\in N_{P'}^+(z)\setminus\{u_{i_q}^+\}$ such that  $u_{i_{q+1}}^- \sim w.$ 
If $w\in V(u_0 \overrightarrow{P'} u_{i_q}^{-2})$, 
then 
$$u_0 \overrightarrow{P'} w^- z u_{i_q} \overrightarrow{P'} u_{i_{q+1}}^- w \overrightarrow{P'} u_{i_q}^- u_{i_{q+1}} \overrightarrow{P'} u_l $$ 
is a Hamilton $(u_0,u_l)$-path in $G$, a contradiction. 
Thus we suppose that 
$w\in V(v_{i_{q+1}}^+ \overrightarrow{P'} u_{l})$. 
Then  
$$u_0 \overrightarrow{P'} u_{i_q}^- u_{i_{q+1}} \overrightarrow{P'} w^- z u_{i_q} \overrightarrow{P'} u_{i_{q+1}}^- w 
\overrightarrow{P'} u_l $$ 
is a Hamilton $(u_0,u_l)$-path in $G$, a contradiction.  
This implies that $u_{i_{q+1}}^-$ is nonadjacent to any vertex of $N_{P'}^+(z)\setminus\{u_{i_q}^+\}.$ 
Combine with Claim~\ref{c1}, we have $e(W)\le 1$, which contradicts the fact that $G$ is a $[k+1,2]$-graph. 
This proves Fact~\ref{f3}.
\end{proof} 
\begin{fact}\label{f4}
\begin{wst}
\item[{\rm (i)}] 
	$ u_{i_q}^{-2}$	is nonadjacent to any vertex of $N_{P'}^-(z)\setminus \{u_{i_q}^{-}, u_{i_{q+1}}^{-}\}$;
\item[{\rm (ii)}] $ u_{i_q}^{-2}\sim u_{i_{q+1}}^{-}$.
\end{wst}
\end{fact}
\begin{proof}[Proof of Fact~\ref{f4}]
(i)  
	Suppose first that there exists an integer $a$ such that $a\le q-1$ 
	and $u_{i_q}^{-2}\sim u_{i_a}^-$. 
	Then let 
	$$Q:=u_0 \overrightarrow{P'} u_{i_a}^- u_{i_q}^{-2}  \overleftarrow{P'}u_{i_a} z u_{i_q} \overrightarrow{P'} u_l.$$
	Clearly, $Q$ is another longest $(u_0,u_l)$-path in $G$ and $u_{i_{q}}^-\not\in V(Q)$.   
	From Fact~\ref{f3}, 
	$u_{i_{q}}^-$ is nonadjacent to any vertex of $V(u_{i_q}^+\overrightarrow{P'} u_{i_{q+1}} )$, this implies that 
	\begin{align*}
		\rho (Q)\ge |u_{i_q}^+\overrightarrow{P'} u_{i_{q+1}}|=|u_{i_q}^+\overrightarrow{P'} u_{i_{q+1}}^-|+1=\rho(P')+1,
	\end{align*}
	which contradicts the choice of $P'$.
	
	Now suppose that there exists an integer $b$ such that $b\ge q+2$  
	and $u_{i_q}^{-2}\sim u_{i_b}^-$.  
	We write 
	$$Q':=u_0 \overrightarrow{P'} u_{i_q}^{-2} u_{i_b}^{-}  \overleftarrow{P'}u_{i_q} z u_{i_b} \overrightarrow{P'} u_l.$$
	Similarly, $Q'$ is another longest $(u_0,u_l)$-path in $G$ and $u_{i_{q}}^-\not\in V(Q')$.   
	Again by Fact~\ref{f3}, 
	\begin{align*}
		\rho (Q')\ge |u_{i_q}^+\overrightarrow{P'} u_{i_{q+1}}|=|u_{i_q}^+\overrightarrow{P'} u_{i_{q+1}}^-|+1=\rho(P')+1,
	\end{align*}
	which contradicts the choice of $P'$. This proves (i).
	
(ii) 
	Let $U:= \{z,  u_{i_q}^{-2}\} \cup N_{P'}^-(z).$ 
	Since $G$ is a $[k+1,2]$-graph, together with (i) and Claim~\ref{c1}, we have
	$ u_{i_q}^{-2}\sim u_{i_{q+1}}^{-}$.
	This completes the proof of Fact~\ref{f4}.
\end{proof}
\begin{fact}\label{f5}
	There exists an integer $p\in \{2,\ldots, k\}\setminus \{q, q+1\}$, 
	such that $u_{i_{q+1}}^{-2} \sim u_{i_p}^-$. 
	Hence $u_{i_q}^{-}\sim u_{i_p}^{-2}$.
\end{fact}
\begin{proof}[Proof of Fact~\ref{f5}] 
	Let $U:= \{u_{i_{q+1}}^{-2}, z\} \cup  N_{P'}^-(z)$.   
	Clearly, $|U|=k+1.$ 
	Since $G$ is a $[k+1,2]$-graph,  
	by Claim~\ref{c1}, there exist two vertices in $N_{P'}^- (z)$ such that both of them are adjacent to $u_{i_{q+1}}^{-2}$. 
	Therefore, there exists at least one vertex in $N_{P'}^- (z)$ distinct from $u_{i_{q+1}}^{-}$,  say $ u_{i_p}^-$, such that $u_{i_{q+1}}^{-2} \sim u_{i_p}^-$, by Fact~\ref{f3}, $p\neq q$, as desired.   

We then prove that $u_{i_q}^{-}\sim u_{i_p}^{-2}$.	
	Suppose to the contrary that $u_{i_q}^{-}\not\sim u_{i_p}^{-2}$.   
	If $p<q$, 
	by Fact~\ref{f4}(ii), $ u_{i_q}^{-2}\sim u_{i_{q+1}}^{-}$, then  we write 
	$$Q:=u_0 \overrightarrow{P'} u_{i_p}^{-}  u_{i_{q+1}}^{-2} \overleftarrow{P'} u_{i_q} 
	z u_{i_p} \overrightarrow{P'} u_{i_q}^{-2}  u_{i_{q+1}}^-  \overrightarrow{P'} u_l.$$
	Clearly, $Q$ is another longest $(u_0,u_l)$-path in $G$ and $u_{i_{q}}^-\not\in V(Q)$.  
	By Claim~\ref{c1}(ii),  
	$u_{i_{q}}^-\not\sim u_{i_p}^{-}.$  
	From Fact~\ref{f3}, 
	$u_{i_{q}}^-$ is nonadjacent to any vertex of $V(u_{i_q}^+\overrightarrow{P'} u_{i_{q+1}} )$, this implies that 
	\begin{align*}
		\rho (Q)\ge |u_{i_q}^+\overrightarrow{P'} u_{i_{q+1}}^{-2}| +
		|\{u_{i_p}^{-2} , u_{i_p}^{-} \}|
		=|u_{i_q}^+\overrightarrow{P'} u_{i_{q+1}}^-|+1=\rho(P')+1,
	\end{align*}
	which contradicts the choice of $P'$. 
	
	We now assume that $p>q+1$.  
	Again by Fact~\ref{f4}(ii),  we let  
	$$Q':=u_0 \overrightarrow{P'} u_{i_q}^{-2} u_{i_{q+1}}^{-}    
	\overrightarrow{P'} u_{i_p}^- u_{i_{q+1}}^{-2} 
	\overleftarrow{P'} u_{i_q} z u_{i_p} 
	\overrightarrow{P'} u_l.$$  
	Then $Q'$ is another longest $(u_0,u_l)$-path in $G$ and $u_{i_{q}}^-\not\in V(Q')$.  
	By Claim~\ref{c1}(ii) and Fact~\ref{f3},
	\begin{align*}
		\rho (Q')\ge |u_{i_q}^+\overrightarrow{P'} u_{i_{q+1}}^{-2}| +
		|\{u_{i_p}^{-2} , u_{i_p}^{-} \}|
		=|u_{i_q}^+\overrightarrow{P'} u_{i_{q+1}}^-|+1=\rho(P')+1,
	\end{align*}
	which contradicts the choice of $P'$, so we do  indeed have $u_{i_q}^{-}\sim u_{i_p}^{-2}$. This proves Fact~\ref{f5}.
\end{proof}
\begin{fact}\label{f6}
	$\rho(P')=2.$
\end{fact}
\begin{proof}[Proof of Fact~\ref{f6}] 
	Suppose to the contrary that $\rho(P')\ge 3.$ 
	Let $U:= \{u_{i_{q+1}}^{-3}, z\}\cup  N_{P'}^- (z)$. Since $G$ is a $[k+1,2]$-graph, by Claim~\ref{c1} and Fact~\ref{f3}, there exist two integers $a,b\in\{ 2,\ldots, k\}\setminus \{q \}$ 
	such that $u_{i_{q+1}}^{-3}\sim u_{i_a}^-$ and  $u_{i_{q+1}}^{-3}\sim u_{i_b}^-$. 
	Then there exists at least one of $\{a,b\}$ is not $p$, without loss of generality, we assume that $a\in\{ 2,\ldots, k\}\setminus \{q,p \}$.  
	Combine with Fact~\ref{f5} gives that there exists a Hamilton $(u_0,u_l)$-path in $G$ (see Tables~\ref{t1} and \ref{t2}), a contradiciton.  This proves Fact~\ref{f6}.
	\begin{table}[h]
		\centering
		\caption{$p<q$}\label{t1}
		\begin{tabular}{@{}lll@{}} 
			\toprule 
			$p<q$ &   
			A Hamilton $(u_0,u_l)$-path in $G$ \\ 
			\midrule 
			$a<p:$ & $u_0 \overrightarrow{P'}  u_{i_a}^- u_{i_{q+1}}^{-3} 
			\overleftarrow{P'} u_{i_p} z u_{i_a} \overrightarrow{P'} u_{i_p}^- u_{i_{q+1}}^{-2}
			\overrightarrow{P'} u_l$ \\
			$p<a<q:$ & $u_0 \overrightarrow{P'} u_{i_p}^{-2}  u_{i_q}^- \overleftarrow{P'} 
			u_{i_a} z u_{i_q} \overrightarrow{P'} u_{i_{q+1}}^{-3} u_{i_a}^- 
			\overleftarrow{P'} u_{i_p}^- u_{i_{q+1}}^{-2} \overrightarrow{P'} u_l $ \\
			$q+1\le a:$ & $u_0 \overrightarrow{P'} u_{i_p}^{-2}  u_{i_q}^- 
			\overleftarrow{P'}  u_{i_p}^- u_{i_{q+1}}^{-2} \overrightarrow{P'} 
			u_{i_a}^- u_{i_{q+1}}^{-3} \overleftarrow{P'} u_{i_q} z u_{i_a} \overrightarrow{P'} 
			u_l$ \\
			\bottomrule 
		\end{tabular}
	\end{table} 
	%
		\begin{table}[h]
		\centering
		\caption{$p>q+1$}\label{t2}
		\begin{tabular}{@{}lll@{}} 
			\toprule 
			$q+1<p$ &   
			A Hamilton $(u_0,u_l)$-path in $G$  \\ 
			\midrule 
			$a<q:$ & $u_0 \overrightarrow{P'}  u_{i_a}^- u_{i_{q+1}}^{-3}  
			\overleftarrow{P'} u_{i_q} z u_{i_a} \overrightarrow{P'} u_{i_q}^- 
			u_{i_p}^{-2} \overleftarrow{P'} u_{i_{q+1}}^{-2} u_{i_p}^-
			\overrightarrow{P'} u_l$ \\
			$q<a<p:$ & $u_0 \overrightarrow{P'} u_{i_q}^- u_{i_p}^{-2} \overleftarrow{P'} 
			u_{i_a} z u_{i_q} \overrightarrow{P'} u_{i_{q+1}}^{-3} u_{i_a}^- 
			\overleftarrow{P'} u_{i_{q+1}}^{-2} u_{i_p}^-
			 \overrightarrow{P'} u_l $ \\
			$ p<a:$ &  $u_0 \overrightarrow{P'}  u_{i_q}^- u_{i_p}^{-2} 
			\overleftarrow{P'} u_{i_{q+1}}^{-2}  u_{i_p}^- 
			\overrightarrow{P'} u_{i_a}^- u_{i_{q+1}}^{-3} 
			\overleftarrow{P'} u_{i_q} z u_{i_a} 
			\overrightarrow{P'} u_l $\\
			\bottomrule 
		\end{tabular}
	\end{table} 
\end{proof}
From Fact~\ref{f6}, one has 
$|u_{i_j}^+\overrightarrow{P'} u_{i_{j+1}}^-|   = 2$ for each $j\in\{1,\ldots, k -1\}$, 
this implies that 
$|G|=3k-1$.  
By the definition of $q$, 
Facts~\ref{f3}, \ref{f4} and \ref{f5} still hold if we replace the subscript $q$ with $j$ for each $j \in \{2, \ldots, k - 1\}$. 

Since $u_{i_2}^{-2} = u_{i_{1}}^+$, by Claim~\ref{c1}, 
$u_{i_{1}}^+$ is nonadjacent to any vertex of $\{z\}\cup N_{P'}^+ (z)$. 
By Fact~\ref{f4}(ii),  $u_{i_2}^{-2}\sim u_{i_{3}}^{-}$, combine with Claim~\ref{c4}(v), 
$u_{i_2}^{-2}$ is nonadjacent to any vertex of $\{u_{i_2}, u_{i_{3}} \}$. 
Together with Fact~\ref{f4}(i), $ u_{i_2}^{-2}$	is nonadjacent to any vertex of $N_{P'}^-(z)\setminus \{u_{i_2}^{-}, u_{i_{3}}^{-}\}$, we have  
$$d_G(u_{i_2}^{-2})\le |  N_{P'}(z)\setminus \{u_{i_2}, u_{i_{3}}\}|  +  |\{u_{i_2}^-, u_{i_3}^-\}| =k-2+2 =k.$$ 
Recall that $G$ is  $k$-connected. Therefore, 
$N_G(u_{i_2}^{-2})= \{u_{i_2}^-, u_{i_3}^-\}\cup N_{P'}(z)\setminus \{u_{i_2}, u_{i_3}\}$.  

If $k\ge 4,$ 
then $ u_{i_{4}}$ exists, 
by a similar discussion, 
$N_G(u_{i_3}^{-2})= \{u_{i_3}^-, u_{i_{4}}^-\}\cup N_{P'}(z)\setminus \{u_{i_3}, u_{i_{4}}\}$.  
Therefore, 
$$u_0  u_{i_3}^{-2} u_{i_3}^{-}  u_{i_2}^{-2} u_{i_2}^{-} u_{i_2} z 
u_{i_3}
\overrightarrow{P'} u_l $$ is a Hamilton $(u_0,u_l)$-path in $G$, a contradiction. 

Hence $k=3,$ then $G$ is $3$-connected. 
By Fact~\ref{f3}, $u_{i_2}^-$ is nonadjacent to any vertex of 
$\{u_{i_2}^+, u_{i_2}^{+2}, u_{i_3}\}$. 
By Claim~\ref{c1}, $u_{i_2}^-\not\sim z$. 
Since $d_G(u_{i_2}^-)\ge 3$,  we have $u_{i_2}^-\sim u_0.$ 
By Fact~\ref{f4}(ii), $ u_{i_2}^{-2}\sim u_{i_{3}}^{-}$.   
It follows that 
$$u_0  u_{i_2}^{-} u_{i_2}^{-2} u_{i_3}^{-} u_{i_3}^{-2} u_{i_2} z u_l$$ 
is a Hamilton $(u_0,u_l)$-path in $G$, a contradiction.  
This completes the proof of Theorem~\ref{thm1}.
\hfill$\Box$
\section{\normalsize Some remarks}\label{s3} 
In this paper, we have proved that 
every $k$-connected $[k+1,2]$-graph is hamiltonian-connected 
except $kK_1\vee G_{k}$, where $k\ge 2$ and $G_{k}$ is an arbitrary graph of order $k.$    
Recently, Professor Xingzhi~Zhan posed the following problem~\cite{Zhan1}. 
\begin{problem}
	What is the minimum size of a connected $[s,t]$-graph of order $n$?
\end{problem}
 We provide a lower bound as follows.
\begin{thm}
  Let $G$ be an $[s,t]$-graph of order $n.$ 
  Then $e(G)\ge \frac{tn(n-1)}{s(s-1)} .$ 
\end{thm}
\begin{proof}
  Using the so-called double counting technique, 
  one has 
  \begin{align*}
    \binom{n-2}{s-2}e(G)=\sum_{\{v_1,\ldots, v_s\}} e( G[\{v_1,\ldots, v_s\}] )
    \ge \binom{n}{s}t,
  \end{align*}
where the summation is over all $s$-element subsets of $V(G)$. 
It implies that $e(G)\ge \frac{tn(n-1)}{s(s-1)}.$ 
\end{proof}

\section*{\normalsize Acknowledgement}  The author is grateful to Professor Xingzhi Zhan for his constant support and guidance. This research  was supported by the NSFC grant 12271170 and Science and Technology Commission of Shanghai Municipality (STCSM) grant 22DZ2229014.

\section*{\normalsize Declaration}

\noindent\textbf{Conflict~of~interest}
The author declares that he has no known competing financial interests or personal relationships that could have appeared to influence the work reported in this paper.

\noindent\textbf{Data~availability}
No data was used for the research described in the article.

\end{document}